\documentclass[a4paper]{amsart}

\usepackage{booktabs}
\usepackage{longtable}
\usepackage{graphicx}
\usepackage{pdflscape}
\usepackage{colortbl}
\usepackage{float}
\usepackage{changepage}

\usepackage[normalem]{ulem}
\usepackage{tikz,xcolor,hyperref}
\definecolor{lime}{HTML}{A6CE39}
\DeclareRobustCommand{\orcidicon}{%
	\begin{tikzpicture}
	\draw[lime, fill=lime] (0,0) 
	circle [radius=0.16] 
	node[white] {{\fontfamily{qag}\selectfont \tiny ID}};
	\draw[white, fill=white] (-0.0625,0.095) 
	circle [radius=0.007];
	\end{tikzpicture}
	\hspace{-2mm}
}

\foreach \x in {A, ..., Z}{%
	\expandafter\xdef\csname orcid\x\endcsname{\noexpand\href{https://orcid.org/\csname orcidauthor\x\endcsname}{\noexpand\orcidicon}}
}

\makeatletter
\@namedef{subjclassname@1991}{$2020$ Mathematics Subject Classification}
\makeatother
\title{Algorithmic Construction of Real Hyperfields from Minimal Axioms}
\author[D.\ E.\ K\k{e}dzierski, K.\ Kuhlmann, H. Stoja\l owska]{D.\ E.\ K\k{e}dzierski \orcidA{} \and K.\ Kuhlmann \orcidB{} \and H.\ Stoja\l owska \orcidC{}}

\thanks{}
\address{}

\address{}
\date{25.08.2025}
\usepackage{geometry} 
\usepackage{tabu}
\usepackage[applemac]{inputenc}
\usepackage[T1]{fontenc}
\usepackage{indentfirst}
\usepackage{diagmac2}
\usepackage{lipsum}
\usepackage{csquotes}
\usepackage{emptypage}
\usepackage{faktor} 
\usepackage{amsmath,amsfonts,amssymb,amsthm}
\usepackage{latexsym}
\usepackage{mathdots}
\usepackage{mathrsfs}  
\usepackage{soul}
\usepackage{algorithm}
\usepackage[noend]{algpseudocode}
\usepackage[T1]{fontenc}
\usepackage{url}
\usepackage{newlfont}

\algblock{Input}{EndInput}
\algnotext{EndInput}
\algblock{Output}{EndOutput}
\algnotext{EndOutput}
\newcommand{\Desc}[2]{\State \makebox[2em][l]{#1}#2}

\newcommand{\HH}{\mathbb{H}}
\newcommand{\KK}{\mathbb{K}}
\newcommand{\SSS}{\mathbb{S}}

\newcommand{\Char}{\text{char}}
\newcommand{\CChar}{\text{C-char}}
\newcommand{\lra}{\longrightarrow}
\renewcommand{\ss}{\underline{s}}

\newcommand{\podwzorem}[2]{\underbrace{#1}\limits_{#2}}

\newtheorem{thm}{Theorem}[section]

\newtheorem{cor}[thm]{Corollary}
\newtheorem{lem}[thm]{Lemma}
\newtheorem{prop}[thm]{Proposition}

\newtheorem{ex}[thm]{Example}

\theoremstyle{definition}
\newtheorem{df}[thm]{Definition}
\newtheorem{rem}[thm]{Remark}

\newcommand{\shift}{\mathrm{shift}}

\begin{document}

\maketitle
\begin{abstract}
We study real hyperfields, focusing in particular on those that are finite with cyclic positive cones. All real hyperfields have characteristic zero, although they can still be classified using the C-characteristic, an invariant that captures essential structural information. We present an algorithm to determine all such hyperfields up to isomorphism and compute their C-characteristic. The algorithm is optimal in the sense that the set of axioms used is minimal. We develop and implement this algorithm in software, enabling a complete classification of finite real hyperfields with cyclic positive cones of order up to 15, as well as identification of the C-characteristic that occur in such hyperfields of order up to 17. Restricting attention to finite hyperfields of cyclic positive cones enables substantial simplification of the algorithm, thereby enhancing its computational efficiency and allowing for the rapid generation of hyperfields of large order. Using a criterion that allows us to determine whether a given finite real hyperfield is a Krasner quotient hyperfield, we obtain many new examples of hyperfields that do not arise from Krasner's quotient construction.
\end{abstract}

\section{Introduction}

Hyperfields are a generalization of fields, equipped with a multivalued addition operation. They were introduced by Krasner in \cite{Krasner1956} as a tool for describing local fields of positive characteristic as limits of local fields of characteristic zero.
Over the years, the theory of hyperfields -- as well as other hyperstructures -- has been actively developed and has found applications in various branches of mathematics, including quadratic form theory \cite{GladkiMarshall2017}, the adele class space \cite{ConnesConsani2011}, valuation theory and model theory \cite{Lee2020, Linzi2023, LinziTouchard2022}, tropical geometry \cite{Lorscheid2022, Viro2010}, algebraic combinatorics \cite{AndersonDavis2019, BakerBowler2019, EppolitoJunSzczesny2020}, and algebraic geometry \cite{Jun2018, Jun2021}. Hyperstructures have also been applied in automata theory and computer science \cite{ChGMassouros2020, GMassouros2003}.

Hyperfields arise naturally via a simple construction. Let $K$ be a field and $H$ a subgroup of its multiplicative group $K^*$. Consider the set of cosets $K/H$ with the natural multiplication and the multivalued addition defined by:
\[
aH + bH = \{cH \mid c \in aH + bH\}.
\]
This construction (for hyperrings) was introduced by Krasner in  \cite{Krasner1983}, and the resulting class of hyperfields is known as Krasner quotient hyperfields. In \cite{ChMassouros1985}, Massouros provided the first example of a hyperfield that cannot be obtained as such a quotient. To this day, determining whether a given hyperfield is a quotient hyperfield remains an open problem in many cases.

In \cite{FrigoLheemLiu2018}, Frigo, Lheem, and Liu studied the characteristic of Krasner quotient hyperfields of the form $\mathbb{F}_p / G$, where $G$ is a subgroup in $\mathbb {F}_p^*$, and computed the characteristic for all such hyperfields with $p < 200$. Unlike in the case of fields, the characteristic of a hyperfield does not need to be a prime number and every natural number greater than 1 can be realized as the characteristic of some infinite hyperfield (see Theorem 3 in \cite{KeLS2023}). However, the question of which characteristics can be realized in finite hyperfields remains open.

In \cite{BakerJin2021}, Baker and Jin provided a complete classification of hyperfields of orders 2, 3, and 4, identifying the number of non-isomorphic classes in each case and determining which numbers arise from Krasner's quotient construction. In \cite{Ameri2020}, Ameri, Eyvazi, and Hoskova-Mayerova developed a computational algorithm that enabled them to classify all hyperfields of orders less than 7 up to isomorphism. They noted that while the algorithm could in principle handle hyperfields of higher orders, it would require substantial computational resources. In \cite{ChMassouros2024} Ch. Massouros and G. Massouros enumerated all hyperfields of order 7 using the fact that the reversibility axiom is deducible in hyperfields from other axioms.

Real hyperfields, introduced by Marshall in \cite{Marshall2006}, are a generalization of real fields (i.e., orderable fields). The theory of real hyperfields has been further developed in several works (see, e.g., \cite{GladkiMarshall2012}, \cite{KLS2022}, \cite{MaxwellSmith2023}).

The set of positive elements of a real hyperfield forms a multiplicative subgroup, called the {\it positive cone}. In this paper, we focus on finite real hyperfields whose positive cones are cyclic subgroups. We present an algorithm for determining all such hyperfields up to isomorphism. The specificity of this class enables us to compute hyperfields of relatively high orders efficiently. In particular, we provide a complete classification of real hyperfields with cyclic positive cones of cardinality up to $13$.

Although all real hyperfields have characteristic zero, they can still be classified using the so-called {\it C-characteristic} (see Definition~\ref{def_char}). We compute the C-characteristics that appear in our class of real hyperfields of cardinality up to $17$. This allows us, via Theorem~\ref{thm_cchar>1}, to identify many new examples of hyperfields that are not Krasner quotient hyperfields.

The paper is organized as follows. In Section~2, we recall the definitions and fundamental properties of hyperfields, real hyperfields, and morphisms between them. Section~3 provides the theoretical foundations of our algorithm for generating finite real hyperfields. In particular, we show that the hyperaddition in a hyperfield with positive cone $P$ is uniquely determined by a family $\{A_x\}_{x \in P}$, where each subset $A_x$ represents the set $1 + x$. Conversely, we present conditions under which a given family of subsets defines a real hyperfield. We also characterize when an isomorphism of the multiplicative groups induces a strict isomorphism of real hyperfields.

In Section~4, we focus on finite real hyperfields with cyclic positive cones. We provide an algorithm to determine such hyperfields and prove its correctness. Furthermore, we establish the conditions under which two hyperfields constructed in this manner are isomorphic.

Section~5 presents the pseudocode of the algorithm described in Section~4. Section~6 contains the results obtained using our generation algorithm. The complete results for real hyperfields of cardinality up to $17$ can be found at: \url{https://zenodo.org/records/16737218}

\section{Preliminaries}

Let $G$ be a nonempty set and $P^*(G)$ the family of nonempty subsets of $G$. A \emph{hyperoperation} is a function 
\[
+ : G \times G \rightarrow P^*(G)
\]
that associates with every pair $(x,y)$ a nonempty subset of $G$ denoted by $x+y$.
For a subset $A \subseteq G$ and $x \in G$ we define
\[
A + x := \bigcup_{a \in A} a+x \hspace{0,5cm} \textnormal{and} \hspace{0,5cm} x + A := \bigcup_{a \in A} x+a.
\]

The notion of a hyperfield was introduced by Krasner in \cite{Krasner1956}. 

\begin{df} \label{def_hyperfield}
A \emph{hyperfield} is a tuple $(H,+, \cdot,0, 1)$, where $+:  H \times H \rightarrow P^*( H)$ is a hyperoperation, $(H \setminus \{ 0 \}, \cdot, 1)$ is an abelian group, and $x \cdot 0 = 0 \cdot x = 0$ for every $x \in  H$, and the following axioms hold:
\begin{itemize}
\item[$(h_1)$] $(x+y)+z=x+(y+z)$ for all $x,y,z \in  H$,
\item[$(h_2)$] $x+y=y+x$ for all $x,y\in  H$,
\item[$(h_3)$] for every $x\in H$ there exists a unique $-x\in H$ such that $0\in x+(-x)$ (the element $-x$ is called an inverse of $x$),
\item[$(h_4)$] $z\in x+y$ implies $x\in -y+z$ for all $x,y,z\in  H$,
\item[$(h_5)$] $z(x+y) = zx+zy$ for all $x,y,z\in  H$.
\end{itemize}
\end{df}
\begin{rem}
    In \cite{ChMassouros2024} Ch. Massouros and G. Massouros showed that axiom $(h_4)$ follows from the other axioms. However, we are not going to omit the axiom $(h_4)$ since we will significantly weaken the remaining ones.
    \end{rem}   
    
The following properties of hyperfields follow from the definition above and will be used freely throughout the rest of the paper.
\begin{prop} \label{prop_hyperfield properties} 
    A hyperfield $H$ has the following properties for every $x, y, z, t \in H$:
    \begin{itemize}
    \item[(i)] $x+0 = \{ x \}=0+x$,
    \item[(ii)] $\bigcup_{x \in  H} (1+x)=  H$,
    \item[(iii)]  $x+y = x(1+x^{-1}y) = (1+xy^{-1})y$ for every $x, y \neq 0$,
    \item[(iv)] $(1+x)+y=(1+y)+x$, 
    \item[(v)]  $z \in x+y$ if and only if $-y \in x+ (-z)$,
    \item[(vi)]  $z \in x+y$ if and only if $x\in z-y$.
    \item[(vii)] $t\Big( \big(x+y\big)+ z\Big)= \big(tx+ty\big)+ tz$.
    \end{itemize}
\end{prop}


Below we present some basic examples of hyperfields. 

\begin{ex} \label{ex_hyperfields}
\mbox{}
\begin{enumerate}
    \item Every field can be turned into a hyperfield if we identify the element $x+y$ with a singleton $\{ x+y\}$. 
    \item Consider the set $K:=\{ 0,1\}$ with the usual multiplication and the hyperaddition $+$ defined as follows:
    \begin{center}
    {\small
\begin{tabu}{|c|[1.3pt]c|c|}
\hline
$ +$ & $0$ &$ 1$ \\ \hline
$ 0$ &$ \{0\}$ & $\{1\}$ \\ \hline
 $1$ & $\{1\}$  & $\{0,1\}$ \\ \hline
\end{tabu} }
\end{center}
Then $\KK: = (K, +, \cdot, 0, 1)$ is a hyperfield called the \emph{Krasner hyperfield}.
\item Consider the set $S:=\{ -1, 0,1\}$ with the usual multiplication and the hyperaddition $+$ defined as follows:
 \begin{center}
 {\small
\begin{tabu}{|c|[1.3pt]c|c|c|}
\hline
 $+$ & $-1$ & $0$ & $1$  \\ \hline
 $-1$ & $\{-1\}$ & $\{-1\}$ & $\{-1,0,1\}$  \\ \hline
 $0$ &$\{-1\}$ & $\{0\}$ & $\{1\}$ \\ \hline
 $1$ &$\{-1,0,1\}$ & $\{1\}$  & $\{1\}$ \\ \hline
\end{tabu} }
\end{center}
Then $\SSS := (S, +, \cdot, 0, 1)$ is a hyperfield called the \emph{hyperfield of signs}.
\item Consider the multiplicative group $H^*=\{ -1,1\}\times \{1,a,a^2\}$. 
Let $H:= H^* \cup \{ 0 \}$. We equip the set $H$ with the following hyperaddition:
\begin{center}
\resizebox{\textwidth}{!}{
\begin{tabu}{|c|[1.3pt]c|c|c|c|c|c|c|}
\hline
    + & $-a^2$ & $-a$ & $-1$ & $0$ & $1$ &$a$ &$a^2$ \\ \tabucline[1.3pt]{-}
    $-a^2$ &$\{-a^2, -1 \}$&$ \{ -a, -1 \} $&$ \{-a^2, -a \} $&$ \{-a^2\}$&$ H^* \setminus \{-1,1\}$&$ H^* \setminus \{ -a^2, a^2 \} $&$ H \setminus \{-a, a \} $\\ \hline
    $-a$ &$\{ -a, -1 \} $&$  \{-a^2, -a \}$&$\{-a^2, -1 \} $ &$ \{-a\} $&$ H^* \setminus \{-a, a \} $&$ H \setminus \{-1,1\}$ & $H^* \setminus \{ -a^2, a^2 \} $\\ \hline
    $-1$ &$\{-a^2, -a \} $&$\{-a^2, -1 \} $&$\{ -a, -1 \} $&$ \{-1\} $& $ H\setminus \{ -a^2, a^2 \}$ &$H^* \setminus \{-a, a \} $ & $ H^* \setminus \{-1,1\}$\\ \hline
    $0$ &$\{-a^2\}$&$\{-a\}$&$\{-1\}$&$\{0\}$&$\{1\}$&$\{a\}$&$\{a^2\}$ \\ \hline
    $1$&$H^* \setminus \{-1,1\}$&$H^* \setminus \{-a, a \} $&$H \setminus \{ -a^2, a^2 \} $ &$\{ 1 \} $&$ \{ 1,a\} $& $\{ 1, a^2\} $ & $\{ a, a^2\} $ \\ \hline
    $a$ &$H^* \setminus \{ -a^2, a^2 \} $&$  H \setminus \{-1,1\} $&$H^* \setminus \{-a, a \}  $&$ \{ a\} $&$\{ 1, a^2\} $&$\{ a, a^2\}  $& $ \{ 1,a\}$ \\ \hline
    $a^2$&$H \setminus \{-a, a \} $&$H^* \setminus \{ -a^2, a^2 \}  $&$H^*\setminus \{-1,1\} $&$\{ a^2\}$&$\{ a, a^2\} $&$  \{ 1,a\} $&$ \{ 1, a^2\}$  \\ \hline
\end{tabu} }
\end{center}
\vspace{0,2cm}
%
%
%
%
%
%
%
which turns the set $H$ to a hyperfield; we denote it by $\HH_1$.

\item The set $H$ from the previous example can be equipped with another hyperaddition:
\begin{center}
\resizebox{\textwidth}{!}{
\begin{tabu}{|c|[1.3pt]c|c|c|c|c|c|c|}
\hline
    + & $-a^2$ & $-a$ & $-1$ & $0$ & $1$ &$a$ &$a^2$ \\ \tabucline[1.3pt]{-}
    $-a^2$ &$\{-a^2, -a \}$&$ \{ -a^2, -1 \} $&$ \{-a, -1 \} $&$ \{-a^2\}$&$ H^* \setminus \{-a^2, a^2\}$&$ H^* \setminus \{ -a, a \} $&$ H \setminus \{-1, 1 \} $\\ \hline
    $-a$ &$\{ -a^2, -1 \} $&$  \{-a,-1 \}$&$\{-a^2, -a \} $ &$ \{-a\} $&$ H^* \setminus \{-1, 1 \} $&$ H \setminus \{-a^2,a^2\}$ & $H^* \setminus \{ -a, a \} $\\ \hline
    $-1$ &$\{ -a,-1 \} $&$\{-a^2, -a \} $&$\{ -a^2, -1 \} $&$ \{-1\} $& $ H\setminus \{ -a, a \}$ &$H^* \setminus \{-1, 1 \} $ & $ H^* \setminus \{-a^2,a^2\}$\\ \hline
    $0$ &$\{-a^2\}$&$\{-a\}$&$\{-1\}$&$\{0\}$&$\{1\}$&$\{a\}$&$\{a^2\}$ \\ \hline
    $1$&$H^* \setminus \{-a^2,a^2\}$&$H^* \setminus \{-1, 1 \} $&$H \setminus \{ -a, a \} $ &$\{ 1 \} $&$ \{ 1,a^2\} $& $\{ a, a^2\} $ & $\{1, a\} $ \\ \hline
    $a$ &$H^* \setminus \{ -a, a \} $&$  H \setminus \{-a^2,a^2\} $&$H^* \setminus \{-1, 1 \}  $&$ \{ a\} $&$\{ a, a^2\} $&$\{1, a\}  $& $ \{ 1,a^2\}$ \\ \hline
    $a^2$&$H \setminus \{-1,1 \} $&$H^* \setminus \{ -a, a \}  $&$H^*\setminus \{-a^2,a^2\} $&$\{ a^2\}$&$\{ 1, a\} $&$  \{ 1,a^2\} $&$ \{ a, a^2\}$  \\ \hline
\end{tabu} }
\end{center}
Then we obtain another hyperfield; we denote it by $\HH_2$. 
\end{enumerate}
\end{ex}

The following lemma is an immediate corollary of a result already noted in \cite{Mittas1973} (p. 369). For the convenience of a reader we present a simple proof.

 \begin{lem} \label{lem_1-1}
    Let $H$ be a hyperfield that is not a field. Then $1-1 \neq \{ 0 \}$.
\end{lem}
\begin{proof}
 Take a hyperfield $H$ which is not a field. Then there is $a \in H$ such that $1+a$ contains at least two different elements, say $x$ and $y$. Then $a \in y-1$  and $x\in 1+a \subseteq 1+(y-1)=y+(1-1)$. If $1-1 = \{ 0 \}$, then $x=y$, a contradiction.
\end{proof}

The first hyperfields, considered by Krasner (see \cite{Krasner1983}), were of a  particular form, so-called \emph{quotient hyperfields}. We present this important construction below.

Let $K$ be a field and $T$ be a subgroup of its multiplicative group $K^*$. For the equivalence relation
\[
x \sim y \text{ if and only if } x=yt \text{ for some } t \in T,
\]
denote the equivalence class of the element $x \in K$ by $[x]_T$ and the set of all equivalence classes by $K_T$. Then the operations
\begin{align*}
    [x]_T+[y]_T& := \{ [x+yt]_T \mid t \in T \} \\
    [x]_T \cdot [y]_T &:= [xy]_T
\end{align*}
turn $K_T$ into a hyperfield called the  \emph{Krasner quotient hyperfield}. \\

At the end of his paper, Krasner asked whether every hyperfield is a quotient hyperfield. Massouros later gave a negative answer by presenting the first non-quotient example \cite{ChMassouros1985}.

\begin{df} \label{def:homo}
Let $H_1$ and $H_2$ be hyperfields.
\begin{enumerate}
    \item A \emph{homomorphism} of hyperfields is a map $\varphi : H_1 \rightarrow H_2$ such that for every $x,y \in H_1$ the following axioms hold:
    \begin{itemize}
        \item[$(f_1)$] $\varphi(0) = 0$,
        \item[$(f_2)$] $\varphi(x \cdot y) = \varphi(x) \cdot  \varphi(y)$,
        \item[$(f_3)$]  $\varphi(x +  y) \subseteq \varphi(x) + \varphi(y)$.
    \end{itemize}
    \item If $\varphi$ satisfies the condition
    \begin{itemize}
        \item[$(f_3')$]  $\varphi(x + y) =\varphi(x) + \varphi(y)$,
    \end{itemize}
    then $\varphi$ is called a \emph{strict homomorphism} of hyperfields.
    \item A strict, bijective homomorphism of hyperfields is called an \emph{isomorphism} of hyperfields.
\end{enumerate}
\end{df}

Below we  present two examples of homomorphisms of hyperfields. 

\begin{ex} \label{ex_homo and iso}
\mbox{}
\begin{enumerate}
    \item Let $K$ be an ordered field and $\SSS$ be the hyperfield of signs described in Example \ref{ex_hyperfields}(3). The map
\begin{align*}
    x & \longmapsto \mathrm{sgn}(x),
\end{align*}
where $\mathrm{sgn}(x)$ denotes the value of the sign function of $x$, is a homomorphism of hyperfields. 
  \item Consider the hyperfields $\HH_1$ and $\HH_2$ from Example \ref{ex_hyperfields}. The map $$a\longmapsto a^2$$  induces an isomorphism of hyperfields. 
\end{enumerate}
\end{ex}

The proof of the following proposition follows in a straightforward way from the definition of a strict homomorphism. In the case of an isomorphism, this was proven in \cite[Prop.~3.23]{Ameri2020}.

\begin{prop}\label{prop:strict_hommorphism_condition}
Let $H_1$ and $H_2$ be hyperfields, and let $\varphi: H_1 \lra H_2$ be a map such that $\varphi(0) = 0$ and $\varphi$ induces a homomorphism of the multiplicative groups $H_1^*$ and $H_2^*$. Then $\varphi$ is a strict homomorphism if and only if
\begin{equation}
    \label{eq:warunke_homomorfizm}
    \varphi(1 + x) = 1 + \varphi(x) \quad \text{for all } x \in H_1.
\end{equation}
\end{prop}

 In \cite{Marshall2006} Marshall introduced the notion of a real hyperfield.
\begin{df}\label{def:real_hyperfield}
The hyperfield $H$ is called a real hyperfield if 
there exists a subset $P \subseteq  H$ such that
    \[
    P+P \subseteq P, \hspace{0,5cm} P\cdot P \subseteq P,  \hspace{0,5cm} P \sqcup -P \sqcup \{ 0 \} =  H.
\]
Then the subset $P \subseteq H$ is called a \emph{positive cone} in $H$.
\end{df}

It was already noted by Marshall (\cite{Marshall2006}) that the set $$\sum \dot H^2 = \{a_1^2+\ldots+a_n^2 \mid n \in \mathbb N, a_i \in H\setminus \{ 0 \}\},$$ is contained in every positive cone of $H$. From this it follows that
$1 \in P$ and $P$ is a subgroup of $H\setminus \{ 0 \}$. The proof of this fact is analogous to the corresponding statement for ordered fields, and we leave it to the reader.

Below we present a few examples of real and non-real hyperfields.
\begin{ex} \label{ex_real hyp}
\mbox{}
\begin{enumerate}
    \item Every real field $K$ with a positive cone $P$ can be turned into a real hyperfield with a positive cone $P$ if we identify the element $x+y$ with the singleton $\{ x+y \}$.
    \item The hyperfield of signs from Example \ref{ex_hyperfields} $(3)$ is a real hyperfield with the positive cone $P:=\{ 1 \}$.
    \item The hyperfields $\HH_1$ and $\HH_2$ from Example \ref{ex_hyperfields} $(4)$ and $(5)$ are real  hyperfields. The positive cone in both cases is the set $P:=\{1, a, a^2\}$.
    \item The Krasner hyperfield $\KK$ from Example \ref{ex_hyperfields} $(2)$ is not a real hyperfield, since $-1 = 1$ in $\KK$. 
\end{enumerate}
\end{ex}

The next proposition shows a structural property of a positive cone in a finite real hyperfield.

\begin{prop}\label{prop:condition_sum_of_all}
    Let $H$ be a finite real hyperfield with positive cone $P$. Then
    \begin{equation}\label{full_condition}
        \bigcup_{x \in P} (1+x) = P.
    \end{equation}
\end{prop}

\begin{proof}
    Let us note that the set $1 - 1$ is symmetric, i.e., 
    \[
        x \in 1 - 1 \quad \iff \quad -x \in 1 - 1.
    \]
    Suppose that there exists $a \in P$ such that $a \notin 1 + x$ for every $x \in P$. This is equivalent to $-P \cap (1 - a) = \emptyset$, therefore $1 - a \subseteq P \sqcup \{0\}$. 
    
    To use induction, we assume that $1 - a^{k-1} \subseteq P \sqcup \{0\}$. Then
    \[
        a^{-1} - a^{k-1} 
        \subseteq a^{-1} - 1 + 1 - a^{k-1} 
        \subseteq a^{-1} \cdot \podwzorem{(1 - a)}{\subseteq P\sqcup \{0\}} + \podwzorem{(1 - a^{k-1})\;}{\subseteq P\sqcup \{0\}},
    \]
    therefore $a^{-1} - a^{k-1} \subseteq P \sqcup \{0\}$.
    
    Thus,
    \[
        1 - a^k = a \cdot \big(a^{-1} - a^{k-1}\big) \subseteq P \sqcup \{0\}
    \]
    for every $k \in \mathbb{N}$.

    Since $P$ is a finite group, every element has finite order. Let $\operatorname{ord}(a) = n$. Then
    \[
        1 - 1 = 1 - a^n \subseteq P \sqcup \{0\}.
    \]
    By the symmetry of the set $1 - 1$, this holds only if $1 - 1 = \{0\}$. From Lemma~\ref{lem_1-1}, we conclude that $H$ is a field, but there is no finite real field.
\end{proof}

\begin{lem} \label{lem_ass cond}

The associativity axiom $(h_1)$ in Definition \ref{def_hyperfield} can be replaced, under the assumption of the remaining axioms, by the identity
\begin{equation} \label{eq_ass}
    (1 + x) + y = (1 + y) + x.
\end{equation}
    for $x, y \in H$. 
\end{lem}
\begin{proof}
   The axioms of Definition \ref{def_hyperfield} obviously imply condition \eqref{eq_ass}. Now assume that axioms $(h_2) - (h_5)$ and condition \ref{eq_ass} hold. Take $a,b,c \in H$. If $b=0$, then 
   $(a+b)+c = a+c = a+(b+c)$. Assume now that $b \neq 0$. Then  
   \begin{align*}
      (a+b)+c = b\big((ab^{-1} +1) +cb^{-1}\big)= b\big((cb^{-1} +1) +ab^{-1}\big)=a+(b+c).
    \end{align*}
 
\end{proof}
\begin{cor}\label{3ass}
For a real hyperfield $H$ with positive cone $P$ condition \eqref{eq_ass} splits into the three following conditions for $a,b \in P$:
    \begin{itemize}
        \item[$(a)$] $(1+a)+b=(1+b)+a$,
        \item[$(b)$] $(1-a)+b = (1+b)-a$,
        \item[$(c)$] $(1-a)-b = (1-b)-a$.
    \end{itemize}
\end{cor}

\begin{prop}
    Let $H$ be a finite real hyperfield with positive cone $P$. Then 
    \[
    (1-a) \cap P \neq \emptyset \quad \text{ and } \quad (1-a) \cap -P \neq \emptyset \quad \textnormal{for all } \quad a\in P.
    \]
\end{prop}
\begin{proof}
Take $a \in P$. Using Proposition \ref{prop:condition_sum_of_all} we obtain $x \in P$ such that $a \in 1+x$. Then $-x \in 1-a$. Hence, $(1-a) \cap -P \neq \emptyset$. There is also $y \in P$ such that $a^{-1} \in 1+y$. Hence, $y \in a^{-1} -1$, so $ay \in 1-a$. We conclude that $(1-a) \cap P \neq \emptyset$.
\end{proof}

\begin{prop} \label{prop_2cond_for_asso}
    Let $H$ be a real hyperfield with positive cone $P$. Then the associativity axiom,  under the assumption of the remaining axioms, can be reduced to the two following conditions for $a,b \in P$:
    \begin{itemize}
        \item[$(a)$] $(1+a)+b=(1+b)+a$,
        \item[$(b_+)$] $\big[(1-a)+b \big]\cap P = \big[(1+b)-a\big]\cap P$.
    \end{itemize}
\end{prop} 

\begin{proof}
First, we show that conditions $(a)$ and $(b_+)$ imply condition $(b)$ of Corollary~\ref{3ass}.

Observe that 
\[
0 \in (1-a)+b \quad \text{if and only if} \quad -b \in 1-a.
\]
The latter is equivalent to $a \in 1+b$, which holds if and only if $0 \in (1+b)-a$.

Now take $y \in P$ such that $-y \in (1+b)-a$. Using condition $(a)$, we obtain
\[
a \in (1+b)+y 
= b\Big( (1+b^{-1}) + yb^{-1} \Big) 
= b\Big( (1+yb^{-1}) + b^{-1} \Big) 
= (b+y)+1.
\]
Hence, there exists $z \in b+y$ such that $a \in 1+z$. Therefore, $y \in z-b$ and $z \in a-1$, so
\[
y \in z-b \subseteq (a-1)-b.
\]
This means that $-y \in (1-a)+b$. This, together with $(b_+)$, proves that
\[
(1-a)+b \supseteq (1+b)-a.
\]

Now take $-y \in (1-a)+b$ for $y \in P$. Hence, $y \in (a-1)-b$. Let $t \in a-1$ be such that $y \in t-b$. Then, using condition $(a)$, we have
\[
a \in t + 1 \subseteq (b+y)+1 
= b \Big( (1+yb^{-1}) + b^{-1} \Big) 
= b \Big( (1+b^{-1}) + yb^{-1} \Big)
= (1+b) + y.
\]
We obtain that $-y \in (1+b)-a$, and thus
\[
(1-a)+b \subseteq (1+b)-a.
\]
This completes the proof that condition $(b)$ is a consequence of conditions $(a)$ and $(b_+)$.\\

Next, we show that condition $(c)$ of Corollary~\ref{3ass} follows from $(a)$ and $(b)$. First observe that
\[
0 \in (1-a)-b \quad \text{if and only if} \quad b \in 1-a.
\]
This is equivalent to $a \in 1-b$, which holds if and only if $0 \in (1-b)-a$.

Take $y \in P$ such that $-y \in (1-a)-b$. Then there is $z \in 1-a$ such that $-y \in z-b$ and, using condition $(b)$, we have
\[
b \in z + y \subseteq (1-a)+y = (1+y)-a.
\]
This means that there is $x \in 1+y$ such that $b \in x-a$, and again using condition $(b)$, we obtain
\[
-y \in 1 - x \subseteq 1 - (a + b) 
= -b \Big( (1 + ab^{-1}) - b^{-1} \Big) 
= -b \Big( (1 - b^{-1}) + ab^{-1} \Big) 
= (1 - b) - a.
\]
This proves that 
\[
\big((1-a)-b\big) \cap -P \subseteq \big((1-b)-a\big) \cap -P.
\]

Assume now that $y \in (1-a)-b \cap P$. Then there is $z \in 1-a$ such that $y \in z-b$. Using condition $(b)$, we obtain
\[
b \in z - y \subseteq (1-a) - y 
= -a \Big( (1 - a^{-1}) + ya^{-1} \Big) 
= -a \Big( (1 + ya^{-1}) - a^{-1} \Big) 
= (-a - y) + 1 = 1 - (a + y).
\]
Hence, there is $x \in a + y$ such that $b \in 1 - x$. So,
\[
y \in x - a \subseteq (1 - b) - a.
\]
This proves that
\[
(1-a)-b \cap P \subseteq (1 - b) - a \cap P.
\]
If we interchange $a$ and $b$, we obtain the reverse inclusions in both cases. Thus, we have shown that the conditions $(a)$ and $(b_+)$ imply the conditions of Corollary~\ref{3ass}. The opposite implication is obvious.

\end{proof}

In the theory of real fields, isomorphisms are expected to preserve positive cones. We place a similar expectation on isomorphisms of real hyperfields.

\begin{df} \label{def_real isom}
    Let $(H_1, P_1)$ and $(H_2, P_2)$ be real hyperfields. An \emph{isomorphism of real hyperfields} is an isomorphism of hyperfields $\varphi : H_1 \lra H_2$ such that
$\varphi(a) \in P_2 \text{ for all } a \in P_1.$

    \end{df}

       
    \begin{ex}
       The isomorphism described in  Example \ref{ex_homo and iso} (2)is an isomorphism of real hyperfields.
    \end{ex}

For hyperfields we distinguish several notions of characteristic. Let us recall two of them (see \cite{Viro2010}).
\begin{df} \label{def_char}
    Let $H$ be a hyperfield.\\[0.5em]
    (a) The smallest positive integer $n$ such that
    \[
    0 \in \underbrace{1 + \ldots + 1}_{n \text{ times}} 
    \]
    is called the \emph{characteristic} of $H$, and we denote it by $\Char\, H = n$. If no such number exists, we put $\Char\, H = 0$.\\[0.5em]
    (b) The smallest positive integer $n$ such that
    \[
    1 \in \underbrace{1 + \ldots + 1}_{n+1 \text{ times}} 
    \]
    is called the \emph{C-characteristic} of $H$, and we denote it by $\CChar\, H = n$. If no such number exists, we put $\CChar\, H = 0$.
\end{df}

\begin{rem}
    The characteristic and the C-characteristic are invariants of an isomorphism of hyperfields. 
\end{rem}

Note that if $H$ is a field, then $\Char\, H = \CChar\, H$. If $H$ is a real hyperfield, then $\Char\, H = 0$, and the argument is analogous to the one used in the case of real fields. However, the C-characteristic may vary. In fact, every natural number can be realized as the C-characteristic of some infinite real hyperfield (see Theorem~4 in \cite{KeLS2023}). Hence, the notion of C-characteristic appears to be a useful tool for classifying real hyperfields. Moreover, the concepts of characteristic and C-characteristic can be applied in the following criterion to determine whether a given hyperfield is not a Krasner quotient hyperfield.

\begin{thm}[Theorem 6, \cite{KeLS2023}] \label{thm_cchar>1}
    Every finite hyperfield $H$ with $\Char\, H = 0$ and $\CChar\, H > 1$ is not a Krasner quotient hyperfield.
\end{thm}

In particular, it follows that every finite real hyperfield $H$ with $\CChar\, H > 1$ is not a Krasner quotient hyperfield.

\section{Generation of Hyperfields and Real Hyperfields}

Consider an abelian group $(G, \cdot, 1)$ and let $H := G \cup \{ 0 \}$, where $0 \notin G$ and $0 \cdot x = x \cdot 0 = 0$ for every $x \in G$. We recall the construction of  hyperfields from  \cite{Ameri2020}.

Consider a family $\mathcal{H}=\{ A_x \}_{x \in H}$ of nonempty subsets of $ H$. Each subset $A_x$ represents the result of the hyperaddition $1+x$.
The family $\mathcal{H}$ induces the hyperaddition $+$ on $ H$ in the following way
\begin{equation}\label{def_hypersum}
   x+y :=x\cdot A_{x^{-1}y} = \{ xt \mid t \in A_{x^{-1}y} \}\quad \textnormal{ for }\quad  x \neq 0\quad \textnormal{ and }\quad  0+y:=\{ y \}. 
\end{equation}
Assume that the family $\mathcal{H}$ satisfies the following conditions:
\begin{itemize}
    \item[$(k_1)$] $A_x+y= A_y+x$, 
    \item[$(k_2)$] $x\cdot A_{x^{-1}y} = A_{xy^{-1}}\cdot y$, for $x,y \neq 0$,
    \item[$(k_3)$]  there exists exactly one element $t \in H$ such that $0\in A_{t}$ (we define $t:= -1$)
    \item[$(k_4)$] for every $x,y \in  H$ the following implication holds 
    \[
    y \in A_x\quad  \Longrightarrow \quad  -x \in A_{-y},
    \]
    where $-x:=-1\cdot x$.
\end{itemize}

\begin{rem}
In the paper \cite{Ameri2020}, the authors considered slightly different conditions. However, the conditions $(k_1)$--$(k_4)$ are equivalent to those used in \cite{Ameri2020}.
\end{rem}

In \cite{Ameri2020}, it was shown that the structure $H$, with the multiplication and hyperaddition defined above, forms a hyperfield.  
For the reader's convenience and to fill in some gaps, we present a proof of this fact.

\begin{thm}\label{thm:construckja_1}
Let $\mathcal{H} = \{ A_x \}_{x \in H}$ be a family of non-empty subsets of $H$ satisfying conditions $(k_1)$--$(k_4)$. Then the structure $(H, +, \cdot, 0, 1)$ is a hyperfield. Moreover, every hyperfield can be constructed in this way.
\end{thm}

\begin{proof}
To prove the second part of the theorem, it suffices to observe that in any hyperfield $H$, the family $\mathcal{H} = \{ A_x \}_{x \in H}$ defined by $A_x = 1 + x$ satisfies conditions $(k_1)$--$(k_4)$. Hence, we now focus on the first part of the theorem.

Let us begin by showing that for every $x \in H$,
\begin{equation} \label{eq:x+0}
    x + 0 = \{ x \} = 0 + x.
\end{equation}

The second equality as well as the case of $x = 0$ follow directly from (\ref{def_hypersum}). Now assume that $x \neq 0$ and let $y \in x + 0$. Using (\ref{def_hypersum}), we obtain $y \in x \cdot A_0$. Hence, $x^{-1}y \in A_0$. By condition $(k_4)$, it follows that $0 \in A_{-x^{-1}y}$, and by $(k_3)$ we have $-x^{-1}y = -1$, which implies $x = y$. Therefore, $x + 0 = \{ x \}$.

The fact that axiom $(h_2)$ holds follows from equation~(\ref{eq:x+0}) for $x = 0$ and from condition $(k_2)$ for $x \neq 0$.
Indeed,
\begin{align*}
    x + y = x \cdot A_{x^{-1}y} = A_{xy^{-1}} \cdot y = y + x.
\end{align*}
Next, let us show that axiom $(h_3)$ is satisfied. If $x = 0$, then equation~(\ref{eq:x+0}) implies that $0 \in x + t$ if and only if $t = 0$, so $-x = 0$ is uniquely determined in this case. Now take $x \in H$ with $x \neq 0$ and assume that $0 \in x + y$, i.e.,
\[
0 \in A_{x^{-1}y}.
\]
From condition~$(k_3)$, we have $x^{-1}y = -1$, hence $y = -x$.

To prove axiom $(h_4)$, take $z \in x + y$. If $x = 0$ or $y = 0$, then axiom $(h_4)$ follows directly from equation~(\ref{eq:x+0}) and axiom $(h_3)$ proven above. If $x, y \neq 0$, then $z \in x \cdot A_{x^{-1}y}$. By condition~$(k_2)$, we have $z \in A_{y^{-1}x} \cdot y$. Thus, $zy^{-1} \in A_{y^{-1}x}$\;.  Now, using condition~$(k_4)$, we obtain $-y^{-1}x \in A_{-zy^{-1}}$, and hence
$
-x \in A_{-zy^{-1}} \cdot y.
$
This implies that
$
x \in A_{-zy^{-1}} \cdot (-y),
$
so finally, $x \in -y + z$.

Next, let us show that axiom $(h_5)$ is satisfied. Take  $x, y, z \in H$. In the case of $z = 0$, the statement is obvious.
Assume now that $z \neq 0$. If $x = 0$, then by equation~(\ref{eq:x+0}) we obtain
\[
z(x + y) = z(0 + y) = \{ zy \} = 0 + zy = z \cdot 0 + zy = zx + zy.
\]

Now suppose that $x \neq 0$. Using (\ref{def_hypersum}), we have
\[
z(x + y) = zx \cdot A_{x^{-1}y}=
zx \cdot A_{(zx)^{-1}zy} = zx + zy,
\]
which completes the proof of axiom $(h_5)$.

Finally, axiom $(h_1)$ follows directly from Lemma~\ref{lem_ass cond} together with condition $(k_1)$. Thus, the proof is complete.
\end{proof}

In \cite{Ameri2020}, the authors used the above construction for the algorithmic classification of hyperfields of cardinality at most 6. Our goal is to identify a minimal family and a set of conditions necessary for the classification of finite real hyperfields.

Let \((P, \cdot, 1)\) be a multiplicative group, and consider the group \(G = \{-1, 1\} \times P = -P \sqcup P\).  
Define \(H := G \sqcup \{0\}\), where \(0 \notin G\), and set \(0 \cdot a = a \cdot 0 = 0\) for every $a \in H$.  

Now, take a family of non-empty subsets \(\mathcal{P} = \{ A_x \}_{x \in P}\), with each \(A_x \subseteq P\). Let \(A_0 = \{1\}\), and for each \(x \in P\), define:
\begin{equation} \label{A-x}
 A_{-x} = \left\{ -y \in -P \cup \{0\} \mid x \in A_y \right\} \cup \left\{ y \in P \mid x^{-1} \in A_{x^{-1}y} \right\}.
\end{equation}

This setup gives us a family \(\mathcal{H} = \{A_x\}_{x \in H}\), which in turn defines a hyperaddition \(+\) on \(H\) using the formula \eqref{def_hypersum}. In this situation, we say that the hyperstructure on \(H\) is \emph{generated} by \(\mathcal{P}\).

\begin{thm} \label{thm_real_construction}
Let \(\mathcal{P} = \{ A_x \}_{x \in P}\) be a family of non-empty subsets of a multiplicative group \(P\), satisfying the following conditions for all \(x, y \in P\):
\begin{itemize}
    \item[$(kr_0)$] $\bigcup\limits_{x\in P}A_x=P$,
    \item[$(kr_1)$] \(A_x + y = A_y + x\),
    \item[$(kr_2)$] \((A_x - y) \cap P = (A_{-y} + x) \cap P\),
    \item[$(kr_3)$] \(x \cdot A_{x^{-1}y} = A_{xy^{-1}}\cdot y\).
\end{itemize}
Then the structure \((H, +, \cdot, 0, 1)\), generated by \(\mathcal{P}\), is a hyperfield. Moreover, every real hyperfield can be built in this way.
\end{thm}

\begin{proof}

Let $\mathcal{H} = \{A_x\}_{x \in H}$ be the family of subsets induced by the family $\mathcal{P} = \{A_x\}_{x \in P}$, as described above, i.e., $A_{-x}$ is given by equation~\ref{A-x} for $x \in P$, and $A_0 = \{1\}$.

By condition $(kr_0)$, it follows that for each $x \in P$ there exists $y \in P$ such that $x \in A_y$. Then, by~\eqref{A-x}, we have $-y \in A_{-x}$, and thus the set $A_{-x}$ is non-empty for all $x \in P$.

We will show that $\mathcal{H}$ satisfies conditions $(k_1)$--$(k_4)$ of Theorem \ref{thm:construckja_1}.

First we will prove the condition $(k_4)$. Assume that $y \in A_x$. If $y \in P \cup \{0\}$, then $A_{-y} = \{ -z \in -P \cup \{0\} \mid y \in A_z \} \cup \{ z \in P \mid y^{-1} \in A_{y^{-1}z} \}$. Therefore, if $x \in P \cup \{0\}$, then $-x \in A_{-y}$ directly from the definition of $A_{-y}$. If $x \in -P$, i.e., $x = -t$ for some $t \in P$, then $y \in A_{-t} \cap P$, which implies that $t^{-1} \in A_{t^{-1}y}$. Thus, we have $1 = t t^{-1} \in t\,A_{t^{-1}y} = y\,A_{y^{-1}t}$, where the last equality follows from $(kr_3)$. Therefore, $y^{-1} \in A_{y^{-1}t}$, and consequently $-x = t \in A_{-y}$. Now suppose that $y \in -P$, then also $x \in -P$ (since $A_x \subseteq P$ for $x \in P$). Assume $x = -t$ for some $t \in P$. We have $y \in A_{-t}$, which, by the definition of $A_{-t}$, implies that $-x = t \in A_{-y}$.

To prove $(k_3)$, assume that $0 \in A_{-x}$ for some $x \in P$ (note that $0 \notin A_{x}$, because $A_x \subseteq P$ for $x \in P \cup \{0\}$). By the definition of $A_{-x}$, this holds if and only if $x \in A_0$, hence $x = 1$.

The condition $(k_2)$ holds when $xy^{-1} \in P$. If $x, y \in P$, then we use $(kr_3)$; and if $x, y \in -P$, then we apply $(kr_3)$ to $-x$ and $-y$. Now, consider the case $xy^{-1} \in -P$. Without loss of generality, we may assume that $x \in P$ and $y \in -P$, i.e., $-y \in P$. Then,
$$
t \in x \cdot A_{x^{-1}y} = x \cdot A_{-(-y)x^{-1}} \Leftrightarrow tx^{-1} \in A_{-(-y)x^{-1}} \Leftrightarrow x(-y)^{-1} \in A_{t(-y)^{-1}} \Leftrightarrow $$
$$-t(-y)^{-1} \in A_{-x(-y)^{-1}} \Leftrightarrow t \in y \cdot A_{xy^{-1}}.
$$
In this argument, the definition of $A_{-z}$ (with $z\in P$) was used twice.

Note that conditions $(kr_1)$ and $(kr_2)$ correspond to conditions $(a)$ and $(b_+)$ in Proposition \ref{prop_2cond_for_asso}, 
which, under the assumption that the remaining axioms hold, are equivalent to condition~$(2)$ in Lemma \ref{lem_ass cond}. 
This latter condition, assuming $(k_2)$--$(k_4)$, is in turn equivalent to $(k_1)$.

For a real hyperfield $H$ with a positive cone $P$, the family of sets $A_x := 1 + x$ for $x \in P$ satisfies conditions $(kr_1)$--$(kr_3)$, and the hyperfield generated by this family coincides with $H$, which follows directly from the axioms of $H$. This proves the second part of the theorem. 
\end{proof}

\begin{rem}
   The set of axioms in Theorem \ref{thm_real_construction} is minimal in the sense that removing any of the axioms results in a structure that is no longer a hyperfield.
\end{rem}

\section{Finite real hyperfields with a cyclic positive cones}
    
Our goal is to optimize both the number of generating sets for finite hyperfields and the number of imposed conditions on them, so that the generation of hyperfields is as computationally efficient as possible. We begin with the following observation.

\begin{prop}
  Let $G$ be a multiplicative group and let $H := G \cup \{0\}$. Let $\mathcal{H} = \{ A_x \}_{x \in H}$ be a family as in Theorem~\ref{thm:construckja_1}. Suppose $V \subseteq H$ satisfies the property:
  \begin{equation} \label{shrink}
    \forall x \in G,\quad x \in V \ \text{or} \ x^{-1} \in V.
  \end{equation}
  Then the family $\mathcal{H}$ is uniquely determined by its subfamily $\{ A_x \}_{x \in V}$.
\end{prop}

\begin{proof}
  If $x \notin V$, then by \eqref{shrink}, we have $x^{-1} \in V$. By condition $(k_2)$ with $y = 1$, it follows that $A_x = x \cdot A_{x^{-1}}$. Hence, each set $A_x$ for $x \in G$ is determined by some $A_{x'}$ with $x' \in V$.
\end{proof}

From now on, we will focus on finite real hyperfields whose positive cone is a cyclic group. In this setting, we obtain the following corollary.

\begin{cor}\label{prop_realhyp_determined_cylic}
Let $H$ be a finite real hyperfield with a cyclic positive cone $P = \langle a \rangle$ of order $N$. Assume that the family $\mathcal{P} = \{ A_{a^i} \}_{i=0}^{N-1}$ generates $H$.  
Then the family $\mathcal{P}$ is uniquely determined by the subfamily
\[
    A_{a^0}, A_{a^1}, \dots, A_{a^K}, \quad \textnormal{where} \quad K = \lfloor N/2 \rfloor.
\]
\end{cor}

In order to define a hyperoperation on the set $H = -P \sqcup \{0\} \sqcup P$, where $P$ is a cyclic group $P = \langle a \rangle$ of order $N$, we follow the algorithm below: 
\begin{itemize}
    \item[\textbf{a.}] Fix the sets $A_{a^l} \subseteq P$ for $l = 0, 1, \dots, K$, where $K = \lfloor N/2 \rfloor$.
    
    \item[\textbf{b.}] Determine the remaining sets $A_{a^k}$ for $k = K+1, K+2, \dots, N-1$ using the formula
    \[
        A_{a^k} = a^k \cdot A_{a^{N-k}}.
    \]
    
    \item[\textbf{c.}] For even $N$, verify that the condition $a^K \cdot A_{a^K} = A_{a^K}$ holds.
    
    \item[\textbf{d.}] Compute the sets $A_{-a^k}$ for $k=0,1,\ldots, N-1$  using formula \eqref{A-x}.
    
    \item[\textbf{e.}] Determine the results of the hyperoperation on $H$ using formula \eqref{def_hypersum}.
\end{itemize}

\begin{thm}
\label{thm:to:daje:hipercialo}
Let $H$ be a hyperstructure determined by the steps \textbf{a.}--\textbf{e.} above. Then condition $(kr_3)$ holds and $H$ is a (real) hyperfield if and only if conditions $(kr_0)$, $(kr_1)$ and $(kr_2)$ of Theorem \ref{thm_real_construction} hold.
\end{thm}

\begin{proof}
It suffices to show that condition $(kr_3)$ holds. Let $x, y \in P$, where $x = a^k$ and $y = a^l$ for some $k, l < N$. Then we compute:
\begin{align*}
    x \cdot A_{x^{-1}y} &= a^k \cdot A_{a^{-k}a^l} = a^k \cdot A_{a^{l-k}} \\
    &= a^k \cdot a^{l-k} \cdot A_{a^{N - (l - k)}} = a^l \cdot A_{a^{k - l}} \\
    &= a^l \cdot A_{a^k a^{-l}} = y \cdot A_{y^{-1}x} = x \cdot A_{xy^{-1}}.
\end{align*}

\end{proof}
Let $H_1$ and $H_2$ be real hyperfields, with positive cones $P_1$ and $P_2$ respectively, such that there exists a (group) isomorphism $\phi: P_1 \rightarrow P_2$. By extending this map via $\phi(0) = 0$ and $\phi(-a) = -\phi(a)$ for all $a \in P_1$, we obtain a bijection $\phi: H_1 \rightarrow H_2$. The following proposition characterizes when $\phi$ is an isomorphism of real hyperfields.

\begin{prop}
Under the assumptions above, $\phi$ is an isomorphism of real hyperfields if and only if  
\[
\phi(1+a) = 1 + \phi(a), \quad \text{for all } a \in P_1.
\]
\end{prop}

\begin{proof}
In view of Proposition \ref{prop:strict_hommorphism_condition}, it suffices to prove that $\phi(1 - a) = 1 - \phi(a)$ for all $a \in P_1$. Let $x \in 1 - a$.

If $x \in P_1 \cup \{0\}$, then:
\[
x \in 1 - a \Leftrightarrow 1 \in x + a \Leftrightarrow 1 \in a (1 + x a^{-1}).
\]
Applying $\phi$, we obtain
\[
1 = \phi(1) \in \phi(a)(1 + \phi(x)\phi(a)^{-1}) = \phi(a) + \phi(x),
\]
which implies $\phi(x) \in 1 - \phi(a)$.

If $x \in -P_1$, then
\[
x \in 1 - a \Leftrightarrow a \in 1 + (-x),
\]
so
\[
\phi(a) \in 1 + \phi(-x) = 1 - \phi(x),
\]
which again implies $\phi(x) \in 1 - \phi(a)$.

Thus, $\phi(1 - a) \subseteq 1 - \phi(a)$. The reverse inclusion follows from the bijectivity of $\phi$.
\end{proof}

Let $H$ be a finite real hyperfield with positive cone $P = \langle a \rangle$, a cyclic group of order $N$. Every automorphism of $P$ is given by
\[
a^m \mapsto a^{mk \bmod N},
\]
where $k$ is coprime to $N$. This, combined with the proposition above, yields the following corollary.

\begin{cor}\label{cor:iso_classes}
Let $H$ be a finite real hyperfield with positive cone $P = \langle a \rangle$ of order $N$, generated by the family $\mathcal{P} = \big\{A_{a^i}\big\}_{0 \leq i < N}$ (as in Theorem \ref{thm:to:daje:hipercialo}). Then the isomorphism class of $H$ is given by
\[
\big[H\big] = \left\{ H_k \mid \gcd(k, N) = 1,\ 1 \leq k < N \right\},
\]
where the hyperfield $H_k$ is generated by the family
\[
\mathcal{P}_k = \big\{A'_{a^i}\big\}_{0 \leq i < N}, \quad \text{with } 
A'_{a^i} = \left\{ a^{mk \bmod N} \mid a^m \in A_{a^{i k^{-1} \bmod N} }\right\},
\]
and $k^{-1}$ denotes the inverse of $k$ modulo $N$.
\end{cor}
\begin{proof}
Assume that $\phi: a^m \mapsto a^{mk \bmod N}$. Then
\[
\phi(A_{a^m}) = \phi(1 + a^m) = 1 + a^{mk \bmod N} = A'_{a^i}
\]
if and only if $i = mk \bmod N$, which implies that $m = i k^{-1} \bmod N$.
\end{proof}

\section{Pseudocode}\label{sec:pseudocode}

In this paper, we consider a real hyperfield ${H}$ with a positive cone $P$ that is a finite cyclic group $\langle a \rangle$.
We will use the following notation:
\begin{itemize}
    \item $N$ denotes the order of the positive cone, i.e., $P = \{1, a, a^2, \dots, a^{N-1}\}$.
    \item $K = \left\lfloor \dfrac{N}{2} \right\rfloor$.
    \end{itemize}
Any subset $S=\{a^{i_1}, a^{i_2}, \dots, a^{i_l}\} \subseteq P$ can be identified with the set of exponents
    \[
        I^+ = \{i_1, i_2, \dots, i_l\} \subseteq \{0,1,2,\dots, N-1\}.
    \]
Next, any nonempty subset $I^+$ of the set $\{0,1,2,\dots, N-1\}$ can be uniquely identified with a number $c \in \{1, \ldots, 2^N-1\}$ via the map:
\begin{equation} \label{formula_num-set}
     I^+ \longmapsto c=\sum_{i\in I^+} 2^i.
\end{equation}
Note that $c$ can be written as a nonzero binary number $(c_{N-1} \ldots c_0)_2$ of $N$ bits and the map
\[
c = (c_{N-1} \cdots c_1 c_0)_2 \longmapsto I^+_c = \big\{ \ell \mid c_\ell = 1 \big\}
\]
is the inverse mapping to \eqref{formula_num-set}.
\begin{ex}
For \( N = 3 \), there are \( 2^3 - 1 = 7 \) nonempty subsets. The correspondence between binary numbers, index sets, and subsets of \( P \) is shown below:
\[
\begin{array}{|c|c|c|c|}
\hline
\text{Decimal } c & \text{Binary } c_2c_1c_0 & I_c & S_c \subseteq \{1, a, a^2\} \\
\hline
1 & 001 & \{0\} & \{1\} \\
2 & 010 & \{1\} & \{a\} \\
3 & 011 & \{0, 1\} & \{1, a\} \\
4 & 100 & \{2\} & \{a^2\} \\
5 & 101 & \{0, 2\} & \{1, a^2\} \\
6 & 110 & \{1, 2\} & \{a, a^2\} \\
7 & 111 & \{0, 1, 2\} & \{1, a, a^2\} \\
\hline
\end{array}
\]    
\end{ex}

Similarly to the set $I^+$, a subset $\{-a^{i_1}, -a^{i_2}, \dots, -a^{i_l}\} \subseteq -P$ is identified with the set
    \[
        I^- = \{i_1, i_2, \dots, i_l\}.
    \]
Once again, each set $I^-$ can be uniquely identified with a number $-c \in \big\{-2^N+1, \ldots, -1\big\}$ via the map:
\begin{equation*} \label{formula_num-set}
     I^- \longmapsto -c=-\sum_{i\in I^-} 2^i.
\end{equation*}
    
    Moreover, we will use the following objects:
    \begin{itemize}
    \item The shift operator is defined by
    \[
        \shift\big(k, I^+\big) := \big\{(k + i) \bmod N \mid i \in I^+\big\},
    \]
    which corresponds to the multiplication of the set $\big\{a^l \mid l \in I\big\}$ by $a^k$.
    \item \texttt{All\_Sets} denotes the array of all nonzero subsets of the form $I_i^+$, where $\texttt{All\_Sets}[i] = I_i^+$ for $i \in \{1, 2, \dots, 2^N - 1\}$.
\end{itemize}

According to Proposition~\ref{prop_realhyp_determined_cylic}, each real hyperfield is uniquely determined by a sequence of nonzero subsets 
\[
\big(1 + a^0,\ 1 + a^1,\ \dots,\ 1 + a^K\big), \quad \text{where } 1 + a^i \subseteq P.
\]
This sequence can be encoded as a $(K+1)$-tuple $\underline{s} = (s_0, s_1, \dots, s_K)$, where each $s_\ell$ denotes the index of the subset $1 + a^\ell$ in the array \texttt{All\_Sets}. We assume that the array \texttt{All\_Sets} is available to all functions and procedures defined in this section. 

Note that when \( N \) is even, i.e., \( N = 2K \), the following identity holds:
\begin{equation}\label{even_condition}
    a^K \cdot \big(1 + a^K\big) = 1 + a^K.
\end{equation}

This condition can be efficiently verified using the function \textsc{Even\_Cond}, which checks whether the set \( 1 + a^K \) is invariant under multiplication by \( a^K \), as implied by equation~\eqref{even_condition}.
\begin{algorithm}[H]
\caption{Verification of condition \eqref{even_condition}}
\begin{algorithmic}[1]
\Function{Even\_Cond}{$\boldsymbol{\underline{s}}$} 
\Comment{Check if \( a^K \cdot (1 + a^K) = 1 + a^K \)}
\Input
  \Desc{$\boldsymbol{\underline{s}}$}{$\quad (K+1)$-tuple \( (s_0, \dots, s_K) \) representing the hyperstructure \( \mathcal{H}(\underline{s}) \)}
\EndInput
\Output
  \Desc{$\boldsymbol{\texttt{true}}$}{\quad if and only if the identity \eqref{even_condition} holds}
\EndOutput
\If{ $\shift\big(K,\ \texttt{All\_Sets}[s_K]\big) = \texttt{All\_Sets}[s_K]$ } 
    \State \Return \texttt{true}
\Else
    \State \Return \texttt{false}
\EndIf
\EndFunction
\end{algorithmic}
\end{algorithm}

To model and manipulate algebraic hyperstructures arising from the cyclic positive cone \( P = \{1, a, a^2, \dots, a^{N-1}\} \), we define a class \texttt{Hyperstructure} that implements the key components and operations required in our computational framework.

The class stores the order \( N \) of the positive cone, and a data structure \( G[\ ] \), which encodes sets of the form \( 1 + a^\ell \), indexed by \( \ell \in \{0, \dots, N-1\} \). The hyperstructure is determined by a $(K+1)$-tuple of integers \( \underline{s} = (s_0, s_1, \dots, s_K) \), provided at construction.

The public interface of the class includes arithmetic methods such as binary and ternary hyper-sums, as well as positive and negative hyper-differences. Additional methods test structural properties (e.g., whether the hyperstructure forms a hyperfield), compute the C-characteristic, and retrieve the defining tuple \( \underline{s} \).

The pseudocode below outlines the design of the \texttt{Hyperstructure} class:
\begin{algorithm}[H]
\caption{Hyperstructure class}
\begin{algorithmic}[1]
\State \textbf{class} Hyperstructure
\State \hspace{1em} \textbf{private:}
\State \hspace{2em}   $\boldsymbol{N}$\quad the order of the cyclic positive cone \( P \)
\State \hspace{2em}   $\boldsymbol{G[\ ]}$\quad an array of length \( N \), where \( G[\ell] \) encodes the set \( 1 + a^\ell \)
\State \hspace{1em} \textbf{public:}
\State \hspace{2em} \textbf{Constructor}($\boldsymbol{N}$: integer,  $\boldsymbol{\underline{s}}$: $(K+1)$-tuple of integers)
\State
\State \hspace{2em} \textbf{Method} \Call{Sum}{$k, l$ : integers}
\Comment{Compute the sum \( a^k + a^l \)}
\State \hspace{2em} \textbf{Method} \Call{Sum}{$k, l, m$ : integers}
\Comment{Compute the sum \( (a^m + a^k) + a^l \)}
\State \hspace{2em} \textbf{Method} \Call{Diff\_Positive}{$k, l$ : integers}
\Comment{Compute the positive part of \( a^k - a^l \)}
\State \hspace{2em} \textbf{Method} \Call{Diff\_Negative}{$k, l$ : integers}
\Comment{Compute the negative part of \( a^k - a^l \)}
\State \hspace{2em} \textbf{Method} \Call{Sum\_PPM\_Pos}{$m, k, l$ : integers}
\Comment{Compute \( ((a^m + a^k) - a^l)_{>0} \)}
\State \hspace{2em} \textbf{Method} \Call{Sum\_PMP\_Pos}{$m, k, l$ : integers}
\Comment{Compute \( ((a^m - a^k) + a^l)_{>0} \)}
\State 
\State \hspace{2em} \textbf{Method} \Call{is\_Hyperfield}{}
\Comment{Check whether the hyperstructure forms a hyperfield}
\State \hspace{2em} \textbf{Method} \Call{C-characteristic}{}
\Comment{Compute the C-characteristic of the hyperstructure}
\State \hspace{2em} \textbf{Method} \Call{get\_tuple}{}
\Comment{Return the defining $(K+1)$-tuple \( \underline{s} \)}
\State
\State \hspace{2em} \textbf{Method} \Call{get\_order}{}
\State \hspace{3em} \Return \( N \)
\State \hspace{2em} \textbf{Operator} \Call{[\ ]}{$i$ : integer} 
\State \hspace{3em} \Return \( G[i] \)
\end{algorithmic}
\end{algorithm}

The following procedure, referred to as the \textit{Constructor}, initializes the table of generator sets \( G[\ ] \) associated with the cyclic positive cone \( P \). Given a natural number \( N \) and a \((K+1)\)-tuple \( \underline{s} = (s_0, s_1, \dots, s_K) \), representing a hyperstructure \( \mathcal{H}(\underline{s}) \), we define \( G \in P^{\times N} \) according to the rules below.

For each index \( 0 \leq i \leq K \), the corresponding set \( G[i] \) is retrieved from a predefined lookup table \texttt{All\_Sets}, indexed by the entries of the tuple \( \underline{s} \). For higher indices \( j > K \), the sets \( G[j] \) are computed by applying a shift operation to previously defined elements, reflecting the symmetry of the structure.
\begin{algorithm}[H]
\caption{Constructor for the generator table \( G[\ ] \)}
\begin{algorithmic}[1]
\State \textbf{Constructor}($\boldsymbol{N}$, $\boldsymbol{\underline{s}}$)
\Input
    \Desc{$\boldsymbol{N}$}{order of the cyclic positive cone \( P \)}
    \Desc{$\boldsymbol{\underline{s}}$}{a $(K+1)$-tuple \( (s_0,\dots, s_K) \) defining the hyperstructure \( \mathcal{H}(\underline{s}) \)}
\EndInput
\For{$i = 0, 1, \dots, K$}
    \State $\boldsymbol{G}[i] \gets \texttt{All\_Sets}[s_i]$
\EndFor
\For{$j = K+1, K+2, \dots, N-1$}
    \State $\boldsymbol{G}[j] \gets \shift\big(j, G[N-j]\big)$
\EndFor
\end{algorithmic}
\end{algorithm}

Before defining a method for verifying associativity, we introduce auxiliary methods for computing sums and differences of elements.

First, the method \textsc{Sum} computes the sum \( a^k + a^l \), while the methods \textsc{Diff\_Positive} and \textsc{Diff\_Negative} compute the positive and negative parts of the difference \( a^k - a^l \), respectively.
\begin{algorithm}[H]
\caption{Method \textsc{Sum} - computing the sum \(a^k + a^l\)}
\begin{algorithmic}[1]
\Function{Sum}{$k, l$}
\Input
    \Desc{$k, l$}{natural numbers such that \(0 \leq k,l < N\)}
\EndInput
\Output
    \Desc{$I^+$}{the set of powers corresponding to the sum \(a^k + a^l\)}
\EndOutput
\If{$k > l$}
    \State $i \gets l - k + N$
\Else
    \State $i \gets l - k$
\EndIf
\State \Return $\shift(k, G[i])$
\EndFunction
\end{algorithmic}
\end{algorithm}
\begin{algorithm}[H]
\caption{Method \textsc{Diff\_Positive} - computing the positive part of the difference \(a^k - a^l\)}
\begin{algorithmic}[1]
\Function{Diff\_Positive}{$k, l$}
\Input
    \Desc{$k, l$}{natural numbers such that \(0 \leq k,l < N\)}
\EndInput
\Output
    \Desc{$I^+$}{the set of powers corresponding to the positive part \(\big(a^k - a^l\big)_{>0}\)}
\EndOutput
\State $I^+ \gets \emptyset$
\ForAll{$k \in \Call{Sum}{i, l}$}
        \State $I^+ \gets I^+ \cup \{i\}$
\EndFor
\State \Return $I^+$
\EndFunction
\end{algorithmic}
\end{algorithm}
\begin{algorithm}[H]
\caption{Method \textsc{Diff\_Negative} - computing the negative part of the difference \(a^k - a^l\)}
\begin{algorithmic}[1]
\Function{Diff\_Negative}{$k, l$}
\Input
    \Desc{$k, l$}{natural numbers such that \(0 \leq k,l < N\)}
\EndInput
\Output
    \Desc{$I^-$}{the set of powers corresponding to the negative part \(\big(a^k - a^l\big)_{<0}\)}
\EndOutput
\State $I^- \gets \emptyset$
\ForAll{$l \in \Call{Sum}{i, k}$}
        \State $I^- \gets I^- \cup \{i\}$
\EndFor
\State \Return $I^-$
\EndFunction
\end{algorithmic}
\end{algorithm}

The following method allows us to compute the sum \((a^m + a^k) + a^l\).
\begin{algorithm}[H]
\caption{Method \textsc{Sum} - computing the sum \((a^m + a^k) + a^l\)}
\begin{algorithmic}[1]
\Function{Sum}{$m, k, l$}
\Input
    \Desc{$m,k,l$}{\quad natural numbers such that \(0 \leq m,k,l < N\)}
\EndInput
\Output
    \Desc{$I^+$}{\quad the set of powers of \((a^m + a^k) + a^l\)}
\EndOutput
\State $I^+ \gets \emptyset$
\ForAll{$i \in \Call{Sum}{m,k}$}
    \State $I^+ \gets I^+ \cup \Call{Sum}{i,l}$
\EndFor
\State \Return $I^+$
\EndFunction
\end{algorithmic}
\end{algorithm}

Next, we define the methods that compute the positive part of \( \big(a^m + a^k\big) - a^l \) and \( \big(a^m - a^k\big) + a^l \).
\begin{algorithm}[H]
\caption{Method \textsc{Sum\_PPM\_Pos} - computing \(\Big(\big(a^m+a^k\big)-a^l\Big)_{>0}\)}
\begin{algorithmic}[1]
\Function{Sum\_PPM\_Pos}{$m,k,l$}
\Input
    \Desc{$m,k,l$}{\quad natural numbers  such that \(0 \leq m,k,l < N\)}
\EndInput
\Output
    \Desc{$I^+$}{\quad the set of powers of the positive part of \(\big(a^m+a^k\big)-a^l\)}
\EndOutput
\State $I^+ \gets \emptyset$
\ForAll{$i \in \Call{Sum}{m,k}$}
        \State $I^+ \gets I^+ \cup \Call{Diff\_Positive}{i,l}$
\EndFor
\State \Return $I^+$
\EndFunction
\end{algorithmic}
\end{algorithm}
\begin{algorithm}[H]
\caption{Method \textsc{Sum\_PMP\_Pos} - computing \(\Big(\big(a^m - a^k\big) + a^l\Big)_{>0}\)}
\begin{algorithmic}[1]
\Function{Sum\_PMP\_Pos}{$m,k,l$}
\Input
    \Desc{$m,k,l$}{\quad natural numbers  such that \(0 \leq m,k,l < N\)}
\EndInput
\Output
    \Desc{$I^+$}{\quad the set of powers of the positive part of \(\big(a^m - a^k\big) + a^l\)}
\EndOutput
\State $I^+ \gets \emptyset$
\ForAll{$i \in \Call{Diff\_Positive}{m,k}$}
    \State $I^+ \gets I^+ \cup \Call{Sum}{i,l}$
\EndFor
\ForAll{$i \in \Call{Diff\_Negative}{m,k}$}
    \State $I^+ \gets I^+ \cup \Call{Diff\_Positive}{l,i}$
    \EndFor
\State \Return $I^+$
\EndFunction
\end{algorithmic}
\end{algorithm}

Now we define the method \textsc{is\_Hyperfield}, which verifies whether a given hyperstructure \(\HH\) is a hyperfield.
\begin{algorithm}[H]
\caption{Method \textsc{is\_Hyperfield} - verification of the hyperfield structure}
\begin{algorithmic}[1]
\Function{is\_Hyperfield}{\phantom{a}}
\Output
    \Desc{\texttt{true}}{\quad if and only if the hyperstructure \(\HH\) satisfies all conditions to be a hyperfield}
\EndOutput
\State $sum \gets G[0]$ \Comment{Step 1: Check condition \((kr_0)\) from Theorem~\ref{thm_real_construction}}
\For{$i = 1,2,\dots,N-1$}
    \State $sum \gets sum \cup G[i]$
\EndFor
\If{$sum \neq \{0, 1, \dots, N-1\}$}
    \State \Return $\boldsymbol{false}$
\EndIf
\Comment{Step 2: Check conditions \((hr_1)\) and \((hr_2)\) from Theorem~\ref{thm_real_construction}}
\For{$k = 0,1,\dots, N-1$}
    \For{$l = 0,1,\dots,N-1$}
        \If{$\Call{Sum}{0, k, l} \neq \Call{Sum}{0, l, k}$}
            \State \Return \texttt{false}
        \EndIf
        \If{$\Call{Sum\_PPM\_Pos}{0, k, l} \neq \Call{Sum\_PMP\_Pos}{0, l, k}$}
            \State \Return \texttt{false}
        \EndIf
    \EndFor
\EndFor
\State \Return \texttt{true}
\EndFunction
\end{algorithmic}
\end{algorithm}

Each hyperfield \(\HH(\underline{s})\) is completely characterized by a \((K+1)\)-tuple \(\underline{s} = (s_0, s_1, \dots, s_K)\).  
For instance, the hyperfields \(\HH_1\) and \(\HH_2\) from Example~\ref{ex_hyperfields} correspond to \(\HH(3,5)\) and \(\HH(5,6)\), respectively, in this notation.

The following methods allow for the reconstruction of \(\underline{s}\) from the hyperfield data.
\begin{algorithm}[H]
\caption{Reconstruction of the \((K+1)\)-tuple defining the hyperstructure}
\label{alg:get_tuple}
\begin{algorithmic}[1]
\Function{get\_number}{$N$, $I$}
    \Comment{Convert a subset $I \subseteq \{0,1,\dots,N-1\}$ to an integer}
    \State $number \gets 0$
    \ForAll{$i \in I$}
        \State $number \gets number + 2^i$
    \EndFor
    \State \Return $number$
\EndFunction
\vspace{0.3em}
\Statex
\Function{get\_tuple}{}
\Output
    \Desc{$\underline{v}$}{$(K+1)$-tuple defining a hyperstructure $\HH(\underline{v})$}
\EndOutput
\State Initialize empty $(K+1)$-tuple $\underline{v}$
\For{$i = 0$ to $K$}
    \State $I^+ \gets G[i]$
    \State $\underline{v}[i] \gets$ \Call{get\_number}{$N, I^+$}
\EndFor
\State \Return $\underline{v}$
\EndFunction
\end{algorithmic}
\end{algorithm}

To classify hyperfields up to isomorphism, we record the generated hyperfields, grouping them according to the C-characteristic (see Definition \ref{def_char}). 

The following method computes the C-characteristic of a hyperfield \(\HH(\underline{s})\), represented by the generating tuple \(\underline{s} = (s_0, s_1, \dots, s_K)\).
\begin{algorithm}[H]
\caption{Methods \textsc{C-characteristic} - computing the \texorpdfstring{$\mathrm{C-char}$}{C-characteristic} of a hyperfield}
\begin{algorithmic}[1]
\Function{C-characteristic}{\phantom{a}} 
\Output
    \Desc{${Cchar}$}{\qquad\quad C-characteristic of the hyperfield $\HH(\underline{s})$}
\EndOutput
\State $I^+ \gets G[0]$
\State $J^+ \gets \emptyset$
\State $Cchar \gets 1$   
\While{$0 \notin I^+$}
    \State $Cchar \gets Cchar + 1$    
    \ForAll{$i \in I^+$} 
            \State $J^+ \gets J^+ \cup G[i]$
    \EndFor
    \State $I^+ \gets J^+$
\EndWhile
\State \Return $Cchar$
\EndFunction
\end{algorithmic}
\end{algorithm}

To determine the isomorphism class of a given hyperfield \(\HH(\underline{s})\), we proceed according to the procedure described in Corollary \ref{cor:iso_classes}. This involves iterating over the elements of the multiplicative group \((\mathbb{Z}/N\mathbb{Z})^\times\) and applying the corresponding automorphisms to obtain all hyperfields isomorphic to \(\HH(\underline{s})\).

The algorithm below computes the isomorphism class of \(\HH(\underline{s})\), returning all \((K+1)\)-tuples corresponding to hyperfields isomorphic to the given one.
\begin{algorithm}[H]
\caption{Computation of the isomorphism class of a hyperfield}
\begin{algorithmic}[1]
\Procedure{Iso\_class}{$\HH$} 
\Input
    \Desc{${\HH}$}{\quad an object of class \texttt{hyperstructure}}
\EndInput
\Output
    \Desc{${Iso[\ ]}$}{\quad array of $(K+1)$-tuples representing all hyperfields isomorphic to $\HH$}
\EndOutput
\State $Iso[0] \gets \HH.\Call{get\_tuple}{}$
\For{$j = 2,3,\dots,N-1$}
    \If{$\gcd(j, N) = 1$}
        \State $j^{-1} \gets$ inverse of $j$ modulo $N$
        \State Initialize empty $(K+1)$-tuple $\underline{v}$
        \For{$i = 0$ to $K$}
            \State $I^+ \gets \emptyset$
            \ForAll{$m \in \HH.G\big[(i \cdot j^{-1}) \bmod N\big]$}
                \State $I^+ \gets I^+ \cup \big\{(m \cdot j) \bmod N\big\}$
            \EndFor
            \State $\underline{v}[i] \gets$ \Call{get\_number}{$N, I^+$}
        \EndFor
        \If{$\underline{v} \notin Iso[\ ]$}
            \State Append $\underline{v}$ to $Iso[\ ]$
        \EndIf
    \EndIf
\EndFor
\State \Return $Iso[\ ]$
\EndProcedure
\end{algorithmic}
\end{algorithm}

Below we present the main algorithm for generating all non-isomorphic real hyperfields with a given cyclic positive cone \(P\) of order \(N\).  
The algorithm iterates over all candidate \((K+1)\)-tuples \(\underline{s}\), where \(K = \lfloor N/2 \rfloor\), that potentially define hyperfields.  
For even \(N\), it first filters tuples using a parity condition (\textsc{Even\_Cond}).  
For each valid tuple \(\underline{s}\), it constructs the corresponding hyperstructure and checks if it satisfies the hyperfield axioms.  
If so, it computes the characteristic \(Cchar\) of the hyperfield and stores the isomorphism class of \(\HH(\underline{s})\) in the corresponding collection \(S[Cchar]\), ensuring no duplicates are saved.  
Each distinct isomorphism class is recorded as a new line in the respective file, facilitating classification and later analysis.
\begin{algorithm}[H]
\caption{Algorithm for generating real hyperfields with a cyclic positive cone}
\begin{algorithmic}[1]
\Input
    \Desc{$\boldsymbol{N}$}{the order of the cyclic positive cone \(P\)}
\EndInput
\Output
    \Desc{$\boldsymbol{S}$}{an array of files, each containing a completed list of \(\lfloor N/2 \rfloor\)-tuples representing all non-isomorphic real hyperfields with positive cone \(P\)}
\EndOutput
\Procedure{generating}{$N$} 
    \State $K \gets \lfloor N/2 \rfloor$
    \If{$N \bmod 2 = 0$}
        \ForAll{$\underline{s} \in \{1, 2, \dots, 2^N - 1\}^{K+1}$}
            \If{\Call{Even\_Cond}{$\underline{s}$} $= \texttt{false}$}
                \State \textbf{continue}
            \EndIf
            \State $\HH \gets$ \textbf{new} Hyperstructure($N, \underline{s}$)
            \If{$\HH.\Call{is\_Hyperfield}{} = \texttt{true}$}
                \State $Cchar \gets$ \Call{C-characteristic}{$\underline{s}$}
                \If{$\underline{s} \notin  S[Cchar]$}
                    \State Append $\Call{Iso\_class}{\HH}$ to $S[Cchar]$ \Comment{Add new isomorphism class to list}
                \State \texttt{write\_line}(file = $S[Cchar]$, line = $\Call{Iso\_class}{\HH}$) \Comment{Save as new line in corresponding file}
                \EndIf
            \EndIf
        \EndFor
    \Else
        \ForAll{$\underline{s} \in \{1, 2, \dots, 2^N - 1\}^{K+1}$}
            \State $\HH \gets$ \textbf{new} Hyperstructure($N, \underline{s}$)
            \If{$\HH.\Call{is\_Hyperfield}{} = \boldsymbol{true}$}
                \State $Cchar \gets$ \Call{C-characteristic}{$\underline{s}$}
                 \If{$\underline{s} \notin  S[Cchar]$}
                   \State Append $\Call{Iso\_class}{\HH}$ to $S[Cchar]$ \Comment{Add new isomorphism class to list}
                \State \texttt{write\_line}(file = $S[Cchar]$, line = $\Call{Iso\_class}{\HH}$) \Comment{Save as new line in corresponding file}
                \EndIf
            \EndIf
        \EndFor
    \EndIf
\EndProcedure
\end{algorithmic}
\end{algorithm}

For the reader's convenience, we present the algorithm that generates the hyperaddition table of the real hyperfield \(\HH(\underline{s})\).

Each subset of the set \(P\) is assigned a number from \(1\) to \(2^N-1\) and each subset of \(-P\) is assigned a number from \(-2^N+1\) to \(-1\), as described in the beginning of this chapter. The empty set is denoted by \(0\).

Thus, any subset \(A \subseteq -P \cup \{0\} \cup P\) can be represented by a triple \((\mathit{neg},\ \mathit{zero},\ \mathit{pos})\), where:
\begin{itemize}
    \item \(\mathit{neg}\) is the number assigned to the subset \(A \cap -P\),
    \item \(\mathit{pos}\) is the number assigned to the subset \(A \cap P\),
    \item \(\mathit{zero} = \texttt{true}\) if \(0 \in A\), and \texttt{false} otherwise.
\end{itemize}

\newpage

We now define the corresponding class that implements this representation.

\begin{algorithm}[H]
\caption{Entries class}
\begin{algorithmic}[1]
\State\textbf{class} Entries
\State \hspace{1em} \textbf{public:}
\State \hspace{2em} $neg$\quad negative number, representing $A\cap -P$  
\State \hspace{2em}  $zero$\quad $\texttt{true}$ if zero belongs to $A$ and $\texttt{false}$ otherwise
\State \hspace{2em} $pos$\quad positive number, representing $A\cap P$  
\vspace{0.1em}
\Statex
\State \hspace{1em} \textbf{Constructor}($n$: integer, $z$: bool, $p$: integer)
\State \hspace{2em} $neg\gets n$, \quad $zero\gets z$, \quad $pos\gets p$
\vspace{0.1em}
\Statex
\State \hspace{1em} \textbf{Method} \Call{swap}{\phantom{a}} \Comment{swap positive with negatice part}
\State \hspace{2em} $temp\gets -neg$, \quad $neg\gets -pos$, \quad $pos\gets temp$
\end{algorithmic}
\end{algorithm}

The following algorithm generates the table of hyperaddition.

\begin{algorithm}[H]
\caption{Generating the Hyperaddition Table of \(\HH(\ss)\)
 }
\begin{algorithmic}[1]
\Input
\Desc{$\boldsymbol{N}$}{the order of the cyclic positive cone $P$}
\Desc{$\boldsymbol{\underline{s}}$}{$(K+1)-$tuple $(s_0,\dots, s_K)$ representing the hyperfield $\HH(\underline{s})$}
\EndInput
\Output
\Desc{$\boldsymbol{Tab[\ ][\ ]}$}{\qquad\quad the $(2N+1)\times (2N+1)-$table of addition for the hyperfield $\HH(\ss)$}
\EndOutput
\Procedure{Table\_of\_hyperaddition}{$\boldsymbol{N}$, $\boldsymbol{\ss}$}
    \State $Tab[N][N]\gets$ \textbf{new}  \Call{Entries}{$0,\texttt{true}, 0$} \Comment{define entry for $0+0$}
    \For {$i=0,1,\dots, N-1$} \Comment{define entries for $-a^k+0$}
        \State $Tab[i][N]\gets$ \textbf{new} \Call{Entries}{$-2^{N-1-i}, \texttt{false}, 0$}
    \EndFor
    \For {$i=1,2,\dots, N$} \Comment{define entries for $0+a^k$}
        \State $Tab[N][i+N]\gets$ \textbf{new}  \Call{Entries}{$0, \texttt{false}, 2^{i-1} $}
    \EndFor
      \State $\HH\gets$ \textbf{new} \Call{Hyperstructure}{N,$\ss$}
    \For {$k=N+1,N+2,\dots, 2N$} \Comment{define entries for $a^{k-N-1}+a^{l-N-1}$}
        \For {$l=k,k+1,\dots, 2N$} 
            \State $I\gets \HH.\Call{Sum}{k-N-1, l-N-1}$
            \State $pos\gets \Call{get\_number}{N,I}$
             \State $Tab[k][l]\gets$ \textbf{new} \Call{Entries}{$0, \texttt{false}, pos$}
        \EndFor
    \EndFor
      \For {$k=0,1,\dots, N-1$} \Comment{define entries for $-a^{N-1-k}-a^{N-1-l}$}
        \For {$l=k,k+1,dots, N-1$} 
         \State $entry\gets Tab[2N-l][2N-k]$ 
             \State entry.\Call{swap}{\phantom{a}}
\algstore{tab}
\end{algorithmic}
\end{algorithm}
\begin{algorithm}                     
\begin{algorithmic} [1]                   
\algrestore{tab}
             \State $Tab[k][l]\gets entry$
        \EndFor
    \EndFor     
      \For {$k=0,1,\dots, N-1$} \Comment{define entries for $-a^{N-1-k}+a^{m-N-1}$}
        \For {$l=N+1,N+2,\dots, 2N$} 
            \State $I^-\gets \HH.$\Call{Diff\_negative}{$l-N-1, N-1-k$}
            \State $I^+\gets \HH.$\Call{Diff\_positive}{$l-N-1, N-1-k$}
            \State $neg \gets$ \Call{get\_number}{$N,I^-$}
            \State $pos \gets$ \Call{get\_number}{$N,I^+$}
            \State $zero \gets \texttt{false} $
            \If{$k+l=2N$} \Comment{In this case we have $-a^s+a^s$}
                \State $zero \gets \boldsymbol{true}$
            \EndIf
            \State $Tab[k][l]\gets$ \textbf{new} \Call{Entries}{$-neg, \texttt{false}, pos$}
        \EndFor
    \EndFor
     \For {$l=0,1,\dots, 2N$} \Comment{setting values symmetrically with respect to the diagonal}
        \For {$k=l+1,l+2,\dots, 2N$} 
            \State $Tab[k][l]\gets Tab[l][k]$ 
        \EndFor
    \EndFor
\EndProcedure
\end{algorithmic}
\end{algorithm}

\section{Classification results}

 This section presents quantitative results on real hyperfields with a cyclic positive cone \(P\). The full list of the algorithm's results can be found at \url{https://zenodo.org/records/16737218}

In the table below, \(N\) denotes the order of the cyclic group \(P\). The cardinality of the corresponding hyperfield with positive cone \(P\) is equal to \(2N + 1\).

 \renewcommand{\arraystretch}{1.3}
\begin{center}
\begin{tabular}{|>{\centering\arraybackslash}p{3.5cm}|>{\centering\arraybackslash}p{1cm}|>{\centering\arraybackslash}p{1cm}|>{\centering\arraybackslash}p{1cm}|>{\centering\arraybackslash}p{1cm}|>{\centering\arraybackslash}p{1.1cm}|>{\centering\arraybackslash}p{1.4cm}|
>{\centering\arraybackslash}p{1.6cm}|} \hline
\rowcolor{gray!60}
\textbf{$\boldsymbol{N}$} & 1 & 2 & 3 & 4 & 5 & 6 & 7\\ \hline
Order of hyperfields & 3 & 5 & 7 & 9 & 11 & 13 & 15 \\ \hline
Cases analyzed by the algorithm & 1 & 9 & 49 & 3375 & 29\,791 & 15\,752\,961 & 260\,144\,641
 \\ \midrule[1.5pt]
Number of hyperfields & 1 & 2 & 11 & 30 & 2\,015 & 49\,321 & 8\,594\,490\\ \hline
Hyperfields (in \%) & 100 & 22.22 & 22.45 & 0.89 & 6.76 & 0.31 & 3.30 \\ \hline
Hyperfields with $\CChar = 1$ & 1 & 2 & 9 & 26 & 1\,469 & 35\,700  & 5\,895\,999 \\ \hline
Hyperfields with $\CChar = 2$ & 0 & 0 & 2 & 4 & 546 & 13\,621  & 2\,698\,059\\  \hline
Hyperfields with $\CChar = 3$ & 0 & 0 & 0 & 0 & 0 & 0 & 432 \\  \midrule[1.5pt]
Number of isomorphism classes & 1 & 2 & 8 & 20 & 521 & 24\,750 & 1\,032\,620\\ \hline
Isomorphism classes (in \%)  & 100 & 22.22 & 16.32 & 0.59 & 1.75 & 0.16 & 0.397 \\ \hline
Isomorphism classes with $\CChar = 1$ & 1 & 2 & 6 & 17 & 380 & 17\,915 & 981\,522\\ \hline
Isomorphism classes with $\CChar = 2$ & 0 & 0 & 2 & 3 & 141 & 6\,835 & 51\,022\\ \hline
Isomorphism classes with $\CChar = 3$ & 0 & 0 & 0 & 0 & 0 & 0 & 76\\ \hline
\end{tabular}
\end{center}
\bigskip

For \(N > 7\), the generation algorithm remains valid; however, the number of real hyperfields increases substantially, rendering the algorithm computationally intensive.  
Consequently, starting from \(N = 8\), we restrict our analysis to real hyperfields whose C-characteristic exceeds \(2\).  
The table below summarizes the count of non-isomorphic hyperfields with C-characteristic equal to \(3\).  
In this case, no hyperfields with higher C-characteristics are found.

 \renewcommand{\arraystretch}{1.3}
\begin{center}
\begin{tabular}{|>{\centering\arraybackslash}p{3.5cm}|>{\centering\arraybackslash}p{1cm}|>{\centering\arraybackslash}p{1cm}|>{\centering\arraybackslash}p{1cm}|>{\centering\arraybackslash}p{1cm}|>{\centering\arraybackslash}p{1cm}|>{\centering\arraybackslash}p{1cm}|
>{\centering\arraybackslash}p{1cm}|>{\centering\arraybackslash}p{1cm}|} \hline
\rowcolor{gray!60}
\textbf{$\boldsymbol{N}$} & 1 & 2 & 3 & 4 & 5 & 6 & 7&8\\ \hline
Order of hyperfields & 3 & 5 & 7 & 9 & 11 & 13 & 15 &17\\ \hline
Isomorphism classes with $\CChar = 3$ & 0 & 0 & 0 & 0 & 0 & 0 & 76 & 190\\ \hline
\end{tabular}
\end{center}\bigskip

The first hyperfield with C-characteristic 2 appears for a positive cone \(P\) of order 3.  An example of such a hyperfield is presented below.

\begin{ex}
    Let $P=\{1,a,a^2\}$ be a positive cone of hyperfield $\HH(6, 3)$. Then the table of hyperaddition has the following form.
    \begin{center}
    \small
\begin{tabular}{|c|c|c|c|c|c|c|c|} \hline
$+$ &$-a^2$&$-a $&$-1 $&$0$&$1 $&$a $&$a^2$\\\hline
$-a^2$&$\{-a,-1\}$&$\{-a^2,-a\}$&$\{-a^2,-1\}$&$\{-a^2\}$&$\{\pm 1,\pm a^2\}$&$\{\pm a,\pm a^2\}$&$\{0,\pm 1,\pm a\}$\\ \hline
$-a $&$\{-a^2,-a\}$&$\{-a^2,-1\}$&$\{-a,-1\}$&$\{-a\}$&$\{\pm 1,\pm a\}$&$\{0,\pm 1, \pm a^2\}$&$
\{\pm a,\pm a^2\}$\\ \hline
$-1 $&$\{-a^2,-1\}$&$\{-a,-1\}$&$\{-a^2,-a\}$&$\{-1\}$&$\{0,\pm a, \pm a^2\}$&$\{\pm 1,\pm a\}$&$
\{\pm 1,\pm a^2\}$\\ \hline
$0$ &$\{-a^2\}$&$\{-a\}$&$\{-1\}$&$\{0\}$&$\{1\}$&$\{a\}$&$\{a^2\}$\\ \hline
$1 $&$\{\pm 1, \pm a^2\}$&$\{\pm 1,\pm a\}$&$\{0,\pm a,\pm a^2\}$&$\{1\}$&$\{a,a^2\}$&$\{1,a\}$&$\{1,a^2\}$\\ \hline
$a $&$\{\pm a,\pm a^2\}$&$\{0,\pm 1,\pm a^2\}$&$\{\pm 1,\pm a\}$&$\{a\}$&$\{1,a\}$&$\{1,a^2\}$&$\{a,a^2\}$\\ \hline
$a^2 $&$\{0,\pm 1,\pm a\}$&$\{\pm a,\pm a^2\}$&$\{\pm 1,\pm a^2\}$&$\{a^2\}$&$\{1,a^2\}$&$\{a,a^2\}$&$\{1,a\}$\\ \hline
\end{tabular}
    \end{center}\medskip
The isomorphism class of this hyperfield is a singleton. 
\end{ex}

The first hyperfield with C-characteristic 3 occurs for a positive cone \(P\) of order 7.  
An example of such a hyperfield is presented below.

\begin{ex}
    Let \( P = \{1, a, \dots, a^6\} \) be a positive cone of the hyperfield \(\HH(104, 61, 27, 30)\). 
    The isomorphism class of this hyperfield contains 6 hyperfields:
\begin{align*}
    \Big[\HH(104, 61, 27, 30)\Big] = \Big\{& \HH(104, 61, 27, 30),\ \HH(104, 45, 91, 60),\ \HH(22, 90, 108, 103), \\
    & \HH(104, 53, 89, 62),\ \HH(22, 106, 110, 99),\ \HH(22, 122, 102, 71) \Big\}
\end{align*}

Below we present part of the table for the hyperaddition of the hyperfield \(\HH(104, 61, 27, 30)\). To complete the missing part of this table, the reader can use the axioms of a hyperfields.
\end{ex}

\clearpage
\thispagestyle{empty}
\begin{landscape}
\begin{adjustwidth}{-4cm}{0cm}
\tiny

\setlength{\tabcolsep}{2pt}
\renewcommand{\arraystretch}{2}

\begin{tabular}{|>{\centering\arraybackslash}p{1cm}|
>{\centering\arraybackslash}p{1cm}|>{\centering\arraybackslash}p{3.17cm}|>{\centering\arraybackslash}p{3.17cm}|>{\centering\arraybackslash}p{3.17cm}|>{\centering\arraybackslash}p{3.28cm}|>{\centering\arraybackslash}p{3.17cm}|>{\centering\arraybackslash}p{3.17cm}|>{\centering\arraybackslash}p{3.17cm}|}
\hline
$+$ & $0$ & $1$ & $a$ & $a^2$ & $a^3$ & $a^4$ & $a^5$ & $a^6$ \\\hline
$-a^6$ & $\{-a^6\}$ & $H^*\!-\!\{-a,-a^2,-a^3,1,a^6\}$ & $\{-a,-a^6,1,\pm a^2,a^4,\pm a^5\}$ & $H^*\!-\!\{-1,-a^2,-a^6,a^3,a^4\}$ & $H^*\!-\!\{-1,-a,a^3,a^4,a^6\}$ & $\{-a^2,-a^5,\pm 1,\pm a^3,a^4,a^6\}$ & $H^*\!-\!\{-a^5,-a^6,1,a,a^2\}$ & $\{0,\pm 1,\pm a,\pm a^3\}$ \\\hline
$-a^5$ & $\{-a^5\}$ & $\{-1,-a^5,\pm a,a^3,\pm a^4,a^6\}$ & $H^*\!-\!\{-a,-a^5,-a^6,a^2,a^3\}$ & $H^*\!-\!\{-1,-a^6,a^2,a^3,a^5\}$ & $\{-a,-a^4,\pm a^2,a^3,a^5,\pm a^6\}$ & $H^*\!-\!\{-a^4,-a^5,1,a,a^6\}$ & $\{0,\pm 1,\pm a^2,\pm a^6\}$ & $H^*\!-\!\{-1,-a,-a^2,a^5,a^6\}$ \\\hline
$-a^4$ & $\{-a^4\}$ & $H^*\!-\!\{-1,-a^4,-a^5,a,a^2\}$ & $H^*\!-\!\{-a^5,-a^6,a,a^2,a^4\}$ & $\{-1,-a^3,\pm a,a^2,a^4,\pm a^5\}$ & $H^*\!-\!\{-a^3,-a^4,1,a^5,a^6\}$ & $\{0,\pm a,\pm a^5,\pm a^6\}$ & $H^*\!-\!\{-1,-a,-a^6,a^4,a^5\}$ & $\{-a^4,-a^6,\pm 1,a^2,\pm a^3,a^5\}$ \\\hline
$-a^3$ & $\{-a^3\}$ & $H^*\!-\!\{-a^4,-a^5,1,a,a^3\}$ & $\{-a^2,-a^6,\pm 1,a,a^3,\pm a^4\}$ & $H^*\!-\!\{-a^2,-a^3,a^4,a^5,a^6\}$ & $\{0,\pm 1,\pm a^4,\pm a^5\}$ & $H^*\!-\!\{-1,-a^5,-a^6,a^3,a^4\}$ & $\{-a^3,-a^5,a,\pm a^2,a^4,\pm a^6\}$ & $H^*\!-\!\{-a^3,-a^4,-a^6,1,a\}$ \\\hline
$-a^2$ & $\{-a^2\}$ & $\{-a,-a^5,1,a^2,\pm a^3,\pm a^6\}$ & $H^*\!-\!\{-a,-a^2,a^3,a^4,a^5\}$ & $\{0,\pm a^3,\pm a^4,\pm a^6\}$ & $H^*\!-\!\{-a^4,-a^5,-a^6,a^2,a^3\}$ & $\{-a^2,-a^4,1,\pm a,a^3,\pm a^5\}$ & $H^*\!-\!\{-a^2,-a^3,-a^5,1,a^6\}$ & $H^*\!-\!\{-a^3,-a^4,1,a^2,a^6\}$ \\\hline
$-a$ & $\{-a\}$ & $H^*\!-\!\{-1,-a,a^2,a^3,a^4\}$ & $\{0,\pm a^2,\pm a^3,\pm a^5\}$ & $H^*\!-\!\{-a^3,-a^4,-a^5,a,a^2\}$ & $\{-a,-a^3,\pm 1,a^2,\pm a^4,a^6\}$ & $H^*\!-\!\{-a,-a^2,-a^4,a^5,a^6\}$ & $H^*\!-\!\{-a^2,-a^3,a,a^5,a^6\}$ & $\{-1,-a^4,a,\pm a^2,\pm a^5,a^6\}$ \\\hline
$-1$ & $\{-1\}$ & $\{0,\pm a,\pm a^2,\pm a^4\}$ & $H^*\!-\!\{-a^2,-a^3,-a^4,1,a\}$ & $\{-1,-a^2,a,\pm a^3,a^5,\pm a^6\}$ & $H^*\!-\!\{-1,-a,-a^3,a^4,a^5\}$ & $H^*\!-\!\{-a,-a^2,1,a^4,a^5\}$ & $\{-a^3,-a^6,1,\pm a,\pm a^4,a^5\}$ & $H^*\!-\!\{-1,-a^6,a,a^2,a^3\}$ \\\hline
$0$ & $\{0\}$ & $\{1\}$ & $\{a\}$ & $\{a^2\}$ & $\{a^3\}$ & $\{a^4\}$ & $\{a^5\}$ & $\{a^6\}$ \\\hline
$1$ & $\{1\}$ & $\{a^3,a^5,a^6\}$ & $\{1,a^2,a^3,a^4,a^5\}$ & $\{1,a,a^3,a^4\}$ & $\{a,a^2,a^3,a^4\}$ & $\{1,a,a^5,a^6\}$ & $\{a,a^2,a^5,a^6\}$ & $\{a,a^2,a^3,a^4,a^6\}$ \\\hline
$a$ & $\{a\}$ & $\{1,a^2,a^3,a^4,a^5\}$ & $\{1,a^4,a^6\}$ & $\{a,a^3,a^4,a^5,a^6\}$ & $\{a,a^2,a^4,a^5\}$ & $\{a^2,a^3,a^4,a^5\}$ & $\{1,a,a^2,a^6\}$ & $\{1,a^2,a^3,a^6\}$ \\\hline
$a^2$ & $\{a^2\}$ & $\{1,a,a^3,a^4\}$ & $\{a,a^3,a^4,a^5,a^6\}$ & $\{1,a,a^5\}$ & $\{1,a^2,a^4,a^5,a^6\}$ & $\{a^2,a^3,a^5,a^6\}$ & $\{a^3,a^4,a^5,a^6\}$ & $\{1,a,a^2,a^3\}$ \\\hline
$a^3$ & $\{a^3\}$ & $\{a,a^2,a^3,a^4\}$ & $\{a,a^2,a^4,a^5\}$ & $\{1,a^2,a^4,a^5,a^6\}$ & $\{a,a^2,a^6\}$ & $\{1,a,a^3,a^5,a^6\}$ & $\{1,a^3,a^4,a^6\}$ & $\{1,a^4,a^5,a^6\}$ \\\hline
$a^4$ & $\{a^4\}$ & $\{1,a,a^5,a^6\}$ & $\{a^2,a^3,a^4,a^5\}$ & $\{a^2,a^3,a^5,a^6\}$ & $\{1,a,a^3,a^5,a^6\}$ & $\{1,a^2,a^3\}$ & $\{1,a,a^2,a^4,a^6\}$ & $\{1,a,a^4,a^5\}$ \\\hline
$a^5$ & $\{a^5\}$ & $\{a,a^2,a^5,a^6\}$ & $\{1,a,a^2,a^6\}$ & $\{a^3,a^4,a^5,a^6\}$ & $\{1,a^3,a^4,a^6\}$ & $\{1,a,a^2,a^4,a^6\}$ & $\{a,a^3,a^4\}$ & $\{1,a,a^2,a^3,a^5\}$ \\\hline
$a^6$ & $\{a^6\}$ & $\{a,a^2,a^3,a^4,a^6\}$ & $\{1,a^2,a^3,a^6\}$ & $\{1,a,a^2,a^3\}$ & $\{1,a^4,a^5,a^6\}$ & $\{1,a,a^4,a^5\}$ & $\{1,a,a^2,a^3,a^5\}$ & $\{a^2,a^4,a^5\}$ \\\hline
\end{tabular}\medskip

\normalsize 
\begin{center}\normalsize
The table of hyperaddition of the hyperfield $\HH(104, 61, 27, 30)$. 
\end{center}

\end{adjustwidth}
\end{landscape}


\end{document}